\documentclass[psamsfonts, reqno, 12pt]{amsart}
\usepackage{amsfonts,amsmath, amsthm, amssymb, latexsym, epsfig}
\usepackage[all]{xy}
\usepackage{xspace}
\usepackage[mathscr]{eucal}

	\advance \topmargin by -\headheight
	\advance \topmargin by -\headsep

	\evensidemargin \oddsidemargin
\RequirePackage{amsfonts}
\RequirePackage{amsmath}
\RequirePackage{amsthm}
\RequirePackage{amssymb}
\RequirePackage{latexsym}
\RequirePackage{epsfig}
\RequirePackage{verbatim}
\RequirePackage[mathscr]{eucal}
\RequirePackage[all]{xy}

	\newcommand{\thin}{ \mbox{H}^2(S, D^*)}

        \theoremstyle{plain}
	\newtheorem{thm}{Theorem}[section]
	\newtheorem{lemma}[thm]{Lemma}
	\newtheorem{theorem}[thm]{Theorem}
	
	\newtheorem{corollary}[thm]{Corollary}

	\theoremstyle{definition}
	\newtheorem{definition}[thm]{Definition}
	
	\newtheorem{remark}[thm]{Remark}
	
	\newtheorem{example}[thm]{Example}

	\numberwithin{equation}{section}
	\numberwithin{figure}{section}

\begin{document}
	\title{Square-Free Rings And Their Automorphism Group}
	\author{Martin Montgomery}
        \address{Department of Mathematics \\
        Piedmont College \\
        165 Central Ave. \\
        Demorest, GA 30535}
        \email{mmontgomery@piedmont.edu}
        \date{\today}
	


	\begin{abstract}

Finite-dimensional square-free $K$-algebras have been completely characterized by Anderson and 
D'Ambrosia as certain semigroup algebras $A \cong K_{\xi} S$ over a square-free semigroup $S$ twisted by some $\xi \in Z^2(S, K^*)$, a two-dimensional cocycle of $S$ with coefficients in the group of units $K^*$ of $K$.   D'Ambrosia extended the definition of square-free to artinian rings with unity and showed every square-free ring has an associated division ring $D$ and square-free semigroup $S$.  We show a square-free ring $R$ can be characterized as a semigroup ring  over a square-free semigroup $S$ twisted by some $(\alpha, \xi) \in Z^2(S, D^*)$, a two-dimensional cocycle of $S$ with coefficients in the nonabelian group of units $D^*$ of a division ring $D$.  Also, to each square-free ring $R \cong D^{\alpha}_{\xi} S$ there exists a short exact sequence
\begin{align*}
1 \longrightarrow H^1_{(\alpha, \xi)} (S, D^*)\longrightarrow  \mbox{Out }R
\longrightarrow  \mbox{Stab}_{(\alpha, \xi)} (\mbox{Aut } S) \longrightarrow 1.
\end{align*}
connecting the outer automorphisms of $R$ to certain cohomology groups related to $S$ and $D$.

	\end{abstract}
        \maketitle
        
 \section*{Introduction}
 The study of outer automorphisms of certain algebraic structures, beginning with Stanley \cite{Stanley} and continued by Scharlau \cite{Scharlau}, Baclawski \cite{Baclawski}, and Coelho \cite{Coelho}, led to a characterization of the outer automorphism group of square-free algebras by Anderson and Dambrosia \cite{And-D'Amb}.  There, Clark's \cite{Clark} cohomology of semigroups, was adapted to prove to every square-free algebra there is an associated short exact sequence relating the first cohomology group with the outer automorphisms of the algebra.  Later, their approach was adopted by Sklar \cite{sklar1} to obtain results about the class of binomial algebras, which includes the class of square-free algebras.
 
Square-free rings were defined by D'Ambrosia \cite{D'Ambrosia} as a generalization of square-free algebras.  She developed the basic properties of square-free rings and provided examples of square-free rings that are not of the form $D \otimes_K A$, where $D$ is a division ring with center $K$ and $A$ is a square-free $K$-algebra.  

Presently, we will combine the above approaches to relate a square-free ring $R$ to an associated short exact sequence with middle term the outer automorphisms of $R$.  Utilizing ideas from non-abelian cohomology (see Dedecker \cite{Dedecker} or Mac Lane \cite{Mac Lane} ) we further extend the cohomology of semigroups from \cite{And-D'Amb}and \cite{sklar1} to a setting related to square-free rings. 

In section 1, we review some of the basic properties of a square-free ring $R$.  These will include the existence of a canonical division ring $D$ as well as an associated square-free semigroup $S$ that plays an important role in the structure of $R$.  Investigating the ring properties of $R$, we start to see a connection to non-abelian cohomology.  In section 2, we develop cohomological ideas useful to characterize the structure of a square-free ring as related to the \textit{thin 2-cohomology} $\thin$ of \cite{Dedecker}.  By making use of the natural action of the semigroup automorphisms $\mbox{Aut }S$ on $\thin$, we completely determine when two square-free rings are isomorphic.  Finally, in section 3 we define several more ideas related to non-abelian cohomology and proceed to show for every square-free ring, there is a short exact sequence with middle term the outer automorphism group of $R$.
 
 \section{Square-free rings}

	Throughout this paper, $R$ will denote a basic indecomposable artinian ring with identity and with Jacobson radical $J=J(R)$.  We will denote by $E=\{e_1, e_2, \ldots, e_n\}$ a basic set of idempotents for $R$ and note the modules
	$$\overline{Re} \cong Re/Je \hspace{.5cm} \mbox{and} \hspace{.5cm} \overline{eR} \cong eR/eJ$$
with $e \in E$ are irredundant sets of the isomorphism classes of simple left and right modules of $R$.
Since $R$ is artinian, each finitely generated $R$-module $M$ has a composition series.  For $e_i \in E$, we denote by
 	$$c_i(M)$$
the multiplicity of the simple modules $\overline{Re_i} \cong Re_i/Je_i$ or $\overline{e_iR} \cong e_iR/e_iJ$ in a composition series for $M$.
 
 A finitely generated indecomposable $R$-module is \textit{square-free} if 
 	$$c_i(M) \leq 1$$
for all $e_i \in E$.  A finitely generated module is square-free if each of its indecomposable direct summands is square-free.  A ring $R$ is left square-free if ${}_R R$ is square-free; similarly, $R$ is right square-free if $R_{R}$ is square-free.  We say $R$ is \textit{square-free} if it both left and right square-free.

By Theorem 1.4 of \cite{D'Ambrosia}, when $R$ is square-free, $e_iRe_i$ is a division ring for all $e_i \in E$ and for all $e_i, e_j \in E$,
	$$ \dim_{e_iRe_i} (e_iRe_j) \leq 1.$$
For $R$ a connected ring, it follows there are ring isomorphisms $ e_i Re_i \cong e_j Re_j$ for all $e_i, e_j \in E$.  Therefore, we speak of \textit{the} division ring (unique up to isomorphism) for a square-free ring $R$ and denote it by $D$.

Associated to any artinian ring is a digraph called the \textit{left quiver}.  The left quiver of $R$ is a directed graph $\Gamma$ with vertex $E$ and $h_{ij}=c_i (Je_j/J^2e_j)$ arrows
\begin{equation*}
\xymatrix{ e_j  \ar[r]_(.45){\alpha_{k}} &  e_i & \hspace{1cm} k=1, \ldots, h_{ij},~~~1 \leq i,~j \leq n.}
\end{equation*}
The \textit{right quiver} of $R$ is defined analogously.  By Lemma 2.1 of \cite{D'Ambrosia}, for $R$ square-free, the left and right quivers are digraph duals.  For convenience, we will focus on the left quiver $\Gamma$ of $R$.  Given any finite directed graph $\Gamma$, there is a natural semigroup associated to $\Gamma$, called the \textit{path semigroup of $\Gamma$}.  Associated to every square-free ring is an \textit{algebra semigroup} called a \textit{square-free semigroup} (see \cite{sklar1} and \cite{D'Ambrosia} for further details) $S$ with zero element $\theta$ and a set of pairwise orthogonal idempotents $E \subseteq S$ such that

\begin{itemize}
\item[(a)] $S=\bigcup_{e_i, e_j \in E} e_i Se_j$;~~\mbox{and}
\item[(b)] $|e_i Se_j \setminus \theta | \leq 1~~\mbox{for each}~~e_i, e_j \in E.$
\end{itemize}
We may view $S$ as a subset of $R$.  To avoid any confusion in the notation, we will denote semigroup multiplication of $s$ and $t$ by $s \cdot t$ and ring multiplication of $s$ and $t$ by $st$.  As a subset of $R$, $S$ is not actually a semigroup, but is \textit{nearly multiplicatively closed} in the following sense; for all $s, t \in S$, there is a unique $d \in D$ with
	$$ st=d (s \cdot t).$$
Since above $d$ is uniquely determined by $s$ and $t$ we will designate it by $\xi(s, t)$ so that
	$$st=\xi(s, t) s \cdot t.$$

Let $R$ be a square-free ring.  For convenience, we often will enumerate the idempotents in $E$ as $e, f, g, \ldots$.  Note by definition of $\xi$, if $s=e \cdot s \cdot f$, it follows then that $\xi(s, f)=1$ and $\xi(e, s)=1$.
Let $s=e \cdot s  \cdot f \in R$.  Since $\mbox{dim}_{eRe} eRf=1=\mbox{dim}_{fRf} eRf$, for any $frf \in fRf$ there exists a unique $er'e \in eRe$ such that
	$$ er'e s=s frf$$
This correspondence defines a function $\alpha: fRf \longrightarrow eRe$.  For a particular $s$ we denote the map by $\alpha_s$.  

\begin{remark} \label{ringcocycle}
Let $s=e s  f$ and $t=f  t g$.  Since $R$ is associative,
	\begin{align*}
	(e s)( t grg)&=(e s)(\alpha_{t} (grg) t)\\ \\
	&=\alpha_{s} (\alpha_{t} 		(grg)) st 
	=\alpha_{s} (\alpha_{t} (grg)) \xi(s, t) s \cdot  t
	\end{align*}
 and
 	\begin{align*}
	(e s)( t grg)&= ( st) grg=(\xi(s, t) s \cdot t) grg\\ \\
	&=\xi(s, t) \alpha_{s \cdot t} (grg) s \cdot t.
	\end{align*}
Since $s, t$ and $grg$ are arbitrary, $\alpha$ has the property that 
	$$\alpha_{s} \circ \alpha_{t} (x) =\xi(s, t) \alpha_{s \cdot t} (x) [\xi(s, t)]^{-1}.$$
\end{remark}
	
Similar to previous treatments of square-free rings (see \cite{D'Ambrosia}), for $d \in D$, we will denote by $\rho_d$ the inner automorphism of  $D$ given by
	$$ \rho_d: x \longrightarrow d x d^{-1}.$$
It follows then that $\alpha$ has the property
	$$\alpha_{s} \circ \alpha_{t}=\rho_{\xi(s, t)} \circ \alpha_{s \cdot t}.$$

 \section{Semigroup Cohomology}\label{cohomology}

In \cite{Clark}, a cohomology with coefficient set in an abelian group was developed for a special class of algebra semigroups.  This cohomology was generalized \cite{And-D'Amb} and later extended \cite{sklar1} to all algebra semigroups.  We now consider the case when the coefficient set is non-abelian.

Let $D$ be a non-abelian group and let $\mbox{Aut}(D)$ be the automorphism group of $D$.  Let $S^{<0>}=E$, and for each $n>0$, define
	$$S^{<n>}=\{ (s_1, s_2, \ldots, s_n) \in S^n |~s_1 \cdot s_2 \cdots s_n \neq \theta \}.$$
For $n \geq 0$ and $G$ an arbitrary (not necessarily abelian) group, let $F^n(S, G)$ denote the set of all functions from $S^{<n>}$ to $G$.  These sets are groups under multiplication and are abelian exactly when $G$ is abelian.

The following set we will designate as \textit{2-cochains}:
	\begin{align*}
	F^1(S, Aut (D)) \times F^2 (S, D^*)
	\end{align*}	

We are especially interested in 2-cochains $(\alpha, \xi)$ which satisfy a ``cocycle identity".
For $s, t, u \in S^{<3>}$, consider all pairs $(\alpha, \xi) \in \mathscr{C}^2 (S,C)$ satisfying the following properties:

	\begin{align}\label{2cochain}
	\alpha_{s} (\xi(t, u) )~ \xi(s, t \cdot u)=\xi(s, t) \xi(s \cdot t, u)
	\end{align}
and
	\begin{align} \label{twistauto}
	\alpha_{s} \circ \alpha_{t}=\rho_{\xi(s, t)} \circ \alpha _{s \cdot t}.
	\end{align}
We call such pairs \textit{2-cocycles} and denote the set of all 2--cocycles by $Z^2(S, D^*)$.  Note by Remark \ref{ringcocycle}, to any square-free ring $R$ there is a pair $(\alpha, \xi)$ satisfying Equation \ref{twistauto}.  Since $R$ is associative, a straightforward check shows Equation \ref{2cochain} is also satisfied.

There is a group action $*$ of  $F^0(S, Aut (D)) \ltimes F^1 (S, D^*)$ on $Z^2 (S, D^{*})$ given by
\begin{align*}
& (\mu, \eta) * (\alpha, \xi)=(\beta, \zeta) \\ \\
& \mu_e \circ \beta_s \circ \mu_f^{-1} =\rho_{\eta(s)} \circ \alpha_s, \hspace{1cm} \mu_e(\zeta(s, t))=\eta(s) \alpha_{s} (\eta(t)) \xi (s, t) \eta( s \cdot t)^{-1}
\end{align*}
for $s=e \cdot s \cdot f$, $t=f \cdot t \cdot g \in S$. 
Altering slightly a definition from \cite{Dedecker}, we call the orbits under $*$ the \textit{thin 2--cohomology} and denote these orbits by
$$\thin .$$

\begin{definition}
We say that two cocycles $(\alpha, \xi)$ and $(\beta, \zeta)$ are \textit{cohomologous} if there is a $(\mu, \eta) \in F^0(S, Aut (D)) \ltimes F^1 (S, D^*)$ such that for all $s, t \in S^*$
\begin{equation} \label{class1}
\mu_e \circ \beta_s \circ \mu_f^{-1}=\rho_{\eta(s)} \circ \alpha_s
\end{equation}
and
\begin{equation}\label{class2}
\mu_e(\zeta(s, t))=\eta(s) \alpha_{s} (\eta(t)) \xi(s, t) \eta(s \cdot t)^{-1}
\end{equation}
for $s=e \cdot s \cdot f$, $t=f \cdot t \cdot g \in S$.  We will write $[\alpha, \xi]=[\beta, \xi]$ for cohomologous cycles.
\end{definition}

We say that a cocyle $(\alpha, \xi) \in Z^2 (S, D^*)$ is \textit{normal} if we have $\alpha_e=1_D$ and $\xi( e, e)=1$ for every $e \in E$.  If $(\alpha, \xi)$ is a 2--cocyle, we may define
$$\eta(s)=\begin{cases} \xi(s, s)^{-1}  & \quad \mbox{if} \quad s \in E \vspace{.2cm}   \cr  1 & \quad \mbox{otherwise} 
\end{cases}$$

Then $(1_D, \eta)*( \alpha,  \xi)$ is a normal cocycle that is cohomologous with $(\alpha, \xi)$.  
Hence, every cohomology class $[\alpha, \xi] \in \thin$ has a normal representative.

Adopting the notation of \cite{D'Ambrosia}, for square-free semigroups $S$ as above, we write
$$R \cong D_{\xi}^{\alpha} S$$
for the left $D$-vector space on $S^*$ with multiplication defined by
$$st=\begin{cases} \xi(s, t)s \cdot t  & \quad \mbox{if} \quad s \cdot t \neq \theta \vspace{.2cm}   \cr  0 & \quad \mbox{otherwise} 
\end{cases}$$
for all $s, t \in S^*$, and
$$ s d=\alpha_s (d)$$
for all $s \in S^*$ and $d \in D$.  Extending the multiplication by linearity gives a square-free ring $R$ with associated division ring $D$ and semigroup $S$.

We can use the group $\mbox{Aut} (S)$ of semigroup automorphisms to define another action on $Z^2 (S, D^*)$.  Given $\phi \in \mbox{Aut} (S)$ and 2-cocycle $(\alpha, \xi) \in Z^2 (S, D^*)$, we can, as in \cite{Anderson}, \cite{D'Ambrosia} and \cite{sklar1}, define a new 2-cocycle $(\alpha^{\phi}, \xi^{\phi}) \in Z^2(S, D^*)$ by
$$\alpha^{\phi} (s)=\alpha (\phi(s))=\alpha_{\phi(s)} \hspace{1cm} \xi^{\phi} (s, t)=\xi( \phi(s), \phi(t))$$
for all $s, t \in S$ with $s \cdot t \neq \theta$.  

If $(\mu, \eta) \in F^0(S, Aut (D)) \ltimes F^1 (S, D^*)$ with
$$(\mu, \eta) * (\alpha, \xi)=(\beta, \zeta)$$
and we define $\mu^{\phi}(s)=\mu_{\phi(e)}$ and $\eta^{\phi}(s)=\eta(\phi(s))$ for $s=e \cdot s $ then
$$ (\mu^{\phi},  \eta^{\phi}) *(\alpha^{\phi}, \xi^{\phi} )= (\beta^{\phi}, \zeta^{\phi})$$
Hence, $[\alpha, \xi]=[\beta, \zeta]$ if and only if $[\alpha^{\phi}, \xi^{\phi}]=[\beta^{\phi}, \zeta^{\phi}]$ and $\mbox{Aut}(S)$ acts on $\thin$.  

We can then extend a result from \cite{Anderson}, \cite{D'Ambrosia} and \cite{sklar1}:

\begin{lemma} \label{rightiso}
Let $S$ be a square-free semigroup, and let $[\alpha, \xi], [\beta, \zeta] \in H^2(S, D^*)$.  If 
	$$[\beta, \zeta] = [\alpha^{\phi}, \xi^{\phi}]$$
for some $\phi \in \mbox{Aut }S$, then as rings, we have
	$$D^{\beta}_{\zeta} S \cong D^{\alpha}_{\xi}S.$$
It follows then, for each $(\alpha, \xi) \in Z^2(S, D^*)$ there exists a normal $(\beta, \zeta) \in Z^2(S, D^*)$ with $D^{\beta}_{\zeta} S \cong D^{\alpha}_{\xi}S$.
\end{lemma}

\begin{proof}
Since $[\beta, \zeta]= [\alpha^{\phi}, \xi^{\phi}] $, there is a pair $(\mu, \eta) \in F^0 (S, \mbox{Aut}(D)) \ltimes F^1(S, D^*)$ such that 
\begin{equation}
\mu_e \circ \beta_s \circ \mu_f^{-1} =\rho_{\eta(s)} \circ \alpha_s^{\phi} 
\end{equation}
and 
\begin{equation}
\mu_e( \zeta(s, t))=\eta(s) \alpha_{s}^{\phi} ( \eta(t)) \xi^{\phi} (s, t)\eta(s \cdot t)^{-1}
\end{equation}
are satisified for $(s, t) \in S^{<2>}$ with $s=e \cdot s \cdot f$.
Define a map $\gamma: D^{\beta}_{\zeta} S \rightarrow D^{\alpha}_{\xi} S$ by
	$$\gamma( d s)=\mu_e (d) \eta(s) \phi(s)$$
and extend linearly.  We claim $\gamma$ is a ring isomorphism.  Certainly, $\gamma$ is one-to-one, onto, and preserves addition.  Let $d_1 s, d_2 t \in D_{\zeta}^{\beta}S$ with $s=e\cdot s \cdot f, t=f\cdot t \cdot g \in S$ and consider $\gamma(d_1 s d_2 t)$:
\begin{align*}
\gamma(d_1 s d_2 t)&=
\gamma(d_1 \beta_s(d_2) s t)=\gamma(d_1 \beta_s(d_2) \zeta(s, t) s \cdot  t) \\ \\
&=\mu_e (d_1 \beta_s(d_2) \zeta(s, t) ) \eta(s \cdot t) \phi(s \cdot t) \\ \\
&=\mu_e (d_1) [\mu_e( \beta_s(d_2))][\mu_e( \zeta(s, t)) ] \eta(s \cdot t) \phi(s \cdot t) \\ \\
&=\mu_e (d_1)[\mu_e \circ \beta_s \circ \mu_f^{-1} ( \mu_f (d_2))] [\eta(s) \alpha^{\phi}_s( \eta(t)) \xi^{\phi}(s, t)\eta(s \cdot t)^{-1}] \eta(s \cdot t) \phi(s \cdot t) \\ \\
&=\mu_e (d_1) [ \eta(s)  \alpha^{\phi}_s ( \mu_f((d_2)) \eta(s)^{-1}] \eta(s) \alpha^{\phi}_s( \eta(t)) \xi^{\phi}(s, t) \phi(s \cdot t) \\ \\
&=\mu_e(d_1) \eta(s) \alpha_{\phi(s)} ( \mu_f(d_2))  \alpha_{\phi(s)}( \eta(t)) \xi (\phi(s), \phi(t)) \phi(s \cdot t)\\ \\
&=\mu_e(d_1) \eta(s) \alpha_{\phi(s)}( \mu_f(d_2))  \alpha_{\phi(s)}( \eta(t))\phi(s ) \phi( t)  \\ \\
&= \mu_e(d_1) \eta(s) \phi(s )  \mu_f(d_2) \eta(t) \phi( t) =\gamma(d_1 s) \gamma(d_2 t)
\end{align*}
Hence, $\gamma$ is a ring isomorphism.
\end{proof}

For $S$ a square-free semigroup and 2-cocycles $( \alpha, \xi),~(\beta, \zeta) \in Z^2( S, D^*)$, we say a ring isomorphism $\gamma: D^{\beta}_{\zeta} S \rightarrow D^{\alpha}_{\xi} S$ is \textit{normal} if $\gamma(E)=E$.  Then as in \cite{sklar1},  we have the following:

\begin{lemma} \label{goodmap}
Let $S$ be a square-free semigroup, let $[ \alpha, \xi]$ and $[\beta, \zeta]$ be normal 2-cocycles in $ Z^2( S, D^*)$, and let $\gamma$ be a normal ring isomorphism from $D^{\beta}_{\zeta} S$ to $D^{\alpha}_{\xi} S$.  Then there exists an automorphism $\phi$ of the semigroup $S$ and   $(\mu, \eta ) \in F^0(S, Aut (D)) \ltimes F^1 (S, D^*)$  such that
	$$ \gamma(ds)=\mu_e(d) \eta(s) \phi (s)$$
for $s=e \cdot s \cdot f \in S^*$.
\end{lemma}

\begin{proof}
First, we show for every $s \in S^*$, there is a unique element $\phi (s) \in S^*$ such that
$\gamma(s) \in D\phi (s)$.  Since $\gamma: D^{\beta}_{\zeta} S \rightarrow D^{\alpha}_{\xi} S$ is a ring isomorphism, we must have
	$$\gamma(s)=\sum_{t \in S^*} d_t t$$
with $d_t \in D$.  But for some $e, f \in E$, $s=e\cdot s \cdot f$ and 
	$$\gamma(s)=\gamma(e) \gamma(s) \gamma(f)=\gamma(e) \left( \sum_{t \in S^*} d_t t \right) \gamma(f).$$
Since $d_t t=t \alpha^{-1}_t (d_t)$ for all $d_t$ and $t$, it follows that
	$$\gamma(s)=\left( \sum_{t \in S^*} d_t   \gamma(e) \cdot  t \cdot  \gamma(f) \right).$$
But since $S$ is square-free and $\gamma$ is an isomorphism, $|  \gamma(e) \cdot  S^* \cdot \gamma(f) | =1$.  Denote by $t_s $ the unique element in $S^*$ such that $ \gamma(e)\cdot  t_s \cdot  \gamma(f)$ whenever $s=e \cdot s \cdot f$ and define $\phi(s)$ by
	$$\phi (s)=t_s$$
and $\phi (\theta)=\theta$.  Since $\gamma$ is a ring isomorphism, it follows that $\phi \in \mbox{Aut }S$.
Also, $\gamma(s) =d_s\phi (s)$ for some unique $d_s \in D$.  Define $\eta: S^* \rightarrow D^*$
by
	$$\eta (s)=d_s.$$
Then $\eta \in F^1 (S, D^*)$.  Now for any $d \in D$, we must have
$$\gamma(d s)=\gamma(d e \cdot s)=\gamma(de)\gamma(s)=\gamma(de)\eta(s) \phi(s) \in D \phi(s)$$
and it follows that $\gamma(de) \in D$.
Define $\mu$ by
$$\mu_e(d)=\gamma(de)$$
for every $e \in E$ and $d$ in $D$.  Again, since $\gamma$ is an isomorphism, it follows that $\mu_e \in \mbox{Aut}(D)$.  Then $\mu \in F^0(S, \mbox{Aut}(D))$ and 
$$ \gamma(ds)=\mu_e(d) \eta(s) \phi(s)$$
as desired.
\end{proof}
Suppose $e \in E$.  As above, $\gamma(e)=\gamma(e^2)=\gamma(e)\gamma(e)$ and the map $\eta \in F^1(S, D^*)$ from Lemma \ref{goodmap} has the property
$$ \eta(e) \phi(e)= \eta(e) \phi(e)=\eta(e) \phi(e) \Rightarrow \eta(e) =\eta(a) \alpha_e( \eta(e))=[\eta(e)]^2.$$
Hence, $\eta(e)=1$ for all $e \in E$.

We now have the following theorem:

\begin{theorem} \label{leftiso}
Let $S$ be a square-free semigroup, and let $[\alpha, \xi],~[\beta, \zeta] \in H^2(S, D^*)$.  Then
	$$ D^{\beta}_{\zeta} S \cong D^{\alpha}_{\xi} S$$
if and only if there exists an automorphism $\phi \in \mbox{Aut} (S)$ such that
	$$[\alpha^{\phi}, \xi^{\phi}]=  [\beta, \zeta] $$

\end{theorem}

\begin{proof}
In Lemma \ref{rightiso}, we showed that $[\alpha^{\phi}, \xi^{\phi}]=  [\beta, \zeta] $ for some
$\phi \in \mbox{Aut}(S)$ gives $D^{\alpha}_{\xi} S \cong D^{\beta}_{\zeta} S$ as rings.

Suppose $D^{\alpha}_{\xi} S  \cong D^{\beta}_{\zeta} S$ via the ring isomorphism $\gamma$, and assume, without loss of generality, that $\xi$ and $\zeta$ are normal.  Since $\gamma$ is an isomorphism, we know that $E$ and $\gamma (E)$ are both basic sets for $ D^{\beta}_{\zeta} S$; therefore, there exists an invertible element $u \in D^{\beta}_{\zeta}S$ such that
	$$\rho_u \circ \gamma (E)=u \gamma(E) u^{-1}=E$$
(see \cite{frank}, Exercise 7.2 (2)).
Then $\rho_u \circ \gamma$ is a normal isomorphism from $D^{\alpha}_{\xi} S$ to $D^{\beta}_{\zeta}S$; thus we may assume without loss of generality that we begin with $\gamma$ a normal map.  By Lemma \ref{goodmap}, there exists $(\mu, \eta) \in   F^0(S, \mbox{Aut}(D)) \ltimes F^1 (S, D^*)$ and an automorphism $\phi \in \mbox{Aut}(S)$ such that
	$$\gamma(ds)=\mu_e( d) \eta(s) \phi(s)$$
for every $s=e \cdot s \cdot f \in S^*$. 
Then for $s, t \in S$ with $s \cdot t \neq \theta$ (and $\phi(s) \cdot \phi(t) \neq \theta)$, we have
\begin{align*}
\mu_e ( \zeta( s,t)) \eta(s \cdot t) \phi(s \cdot t) &=
\gamma( \zeta(s ,t ) s \cdot t) \\ \\
&= \gamma( s t)= \gamma(s) \gamma(t) \\ \\
&=\eta(s) \phi(s) \eta(t) \phi(t) \\ \\
&= \eta(s) \alpha^{\phi}_s (\eta(t)) \xi^{\phi} (s, t) \phi(s \cdot t)
\end{align*}
Equating coefficients in $D$ gives
\begin{equation} \label{result1}
\mu_e ( \zeta (s, t))=\eta(s) \alpha^{\phi}_s (\eta(t)) \xi^{\phi} (s, t) \eta( s \cdot t)^{-1}.
\end{equation} 

For $s=e \cdot s \cdot f \in S$ and $t=f \cdot t \in S$ such that $s \cdot t \neq \theta$, consider the arbitrary element of $D$ denoted by $\mu_f^{-1} ( d)$ .  Then
\begin{align*}
 \gamma( s (\mu_f^{-1}(d)f))=\gamma( \beta_s ( \mu_f^{-1}(d)) s) 
\Rightarrow & \gamma(s) \gamma(\mu_f^{-1}(d)f) =\mu_e ( \beta_s ( \mu_f^{-1} (d))) \eta(s) \phi(s) \\ \\
\Rightarrow & \eta(s) \phi(s) d =\mu_e ( \beta_s ( \mu_f^{-1} (d))) \eta(s) \phi(s) \\ \\
\Rightarrow & \eta(s) \alpha^{\phi}_s (d) =\mu_e ( \beta_s ( \mu_f^{-1} (d))) \eta(s)
\end{align*}
It follows then that
\begin{equation}\label{result2}
 \mu_{e} \circ \beta_{s} \circ \mu_{f }^{-1}= \rho_{\eta(s)} \circ \alpha^{\phi}_s
\end{equation}
By Equations \ref{result1} and \ref{result2}, we have
$$(\mu,\eta) * (\alpha^{\phi}, \xi^{\phi})=(\beta, \zeta).$$
Hence $[ \alpha^{\phi}, \xi^{\phi}]= [\beta, \zeta] \in \thin$, as desired.
\end{proof}

\pagebreak

\section{The Automorphism Group of a Square-Free Ring}

Let $R$ be a binomial ring.  By Lemma \ref{rightiso} and Theorem \ref{leftiso}, we may assume $R \cong D^{\alpha}_{\xi} S$, where $S$ is a square-free semigroup with idempotent set $E$, $D$ is the canonical division ring related to $R$ and $(\alpha, \xi) \in Z^2(S, D^*)$ is a normal 2-cocycle.  We will now consider the outer automorphism group
$$\mbox{Out} R  =\mbox{Aut }R/\mbox{Inn } R$$
of $R$, where $\mbox{Aut } R$ is the group of all ring automorphisms and $\mbox{Inn}(R)$ is the normal subgroup of $\mbox{Aut } R$ consisting of maps $\rho_r$ for invertible $r \in R$.  Finally, $1_D$ and $1_S$, respectively, will each denote the identity automorphism of the ring $D$ and the semigroup $S$.  An unlabeled $1$ will denote the multiplicative identity of $D$.

\begin{definition}
Let $(\alpha, \xi) \in Z^2(S, D^*)$.  A pair $(\mu, \eta) \in F^0(S, \mbox{Aut}(D)) \ltimes F^1(S, D^*)$ is called an \textit{$(\alpha, \xi)$ 1-cocyle} if
\begin{equation}
\mu_e \circ \alpha_s \circ \mu_f= \rho_{\eta(s)} \circ \alpha_s
\end{equation}
and
\begin{equation}
\mu_e( \xi(s, t))= \eta(s) \alpha_s( \eta(t)) \xi(s,t) \eta( s \cdot t)^{-1}
\end{equation}
for all $s=e \cdot s \cdot f \in S$ and $(s, t) \in S^{<2>}$.  The set of all $(\alpha, \xi)$ 1-cocyles will be denoted by $Z^1_{ (\alpha, \xi)} (S, D^*)$.
\end{definition}

\begin{remark}
Note that  $Z^1_{ (\alpha, \xi)} (S, D^*)$ is the stabilizer of $(\alpha, \xi)$ under the action $*$ of $ F^0(S, \mbox{Aut}(D)) \ltimes F^1(S, D^*)$ on $Z^2 (S, D^*)$.  Hence $Z^1_{ (\alpha, \xi)} (S, D^*)$ is a subgroup of $F^0(S, \mbox{Aut}(D)) \ltimes F^1(S, D^*)$.  It easily follows that $Z^1_{ (\beta, \zeta)} (S, D^*) \cong Z^1_{ (\alpha^{\phi}, \xi^{\phi})} (S, D^*)$ whenever $[\beta, \zeta]=[\alpha^{\phi}, \xi^{\phi}]$ in $\thin$.
\end{remark}

There is a group action $\star: F^0 (S, D^*) \times Z^1_{ (\alpha, \xi)} (S, D^*) \rightarrow Z^1_{ (\alpha, \xi)} (S, D^*)$ given by
\begin{align*}
& \epsilon \star (\mu, \eta)=(\hat{\mu}, \hat{\eta}) \\ \\
& (\hat{\mu}(s))(d)= (\rho_{\epsilon(e)} \circ \mu_s)(d) \hspace{1cm} \hat{\eta}(s)=\epsilon(e) \eta(s) \alpha_s ( \epsilon_f^{-1})
\end{align*}
for $d \in D$ and $s=e \cdot s \cdot f$.  

\begin{definition}
We call the orbit of $(1_D, 1)$ under this action the set of $(\alpha, \xi)$  \textit{1-coboundaries} and denote it by $B^1_{ (\alpha, \xi)} (S, D^*)$.  An easy check shows $B^1_{ (\alpha, \xi)} (S, D^*)$ is a normal subgroup of $Z^1_{ (\alpha, \xi)} (S, D^*)$.  We define the $(\alpha, \xi)$ 1-cohomology group of $S$ with coefficients in $D^*$ to be the factor group
$$H^1_{ (\alpha, \xi)} (S, D^*)=Z^1_{ (\alpha, \xi)} (S, D^*)/B^1_{ (\alpha, \xi)} (S, D^*).$$
\end{definition}

For $(\mu, \eta) \in Z^1_{(\alpha, \xi)}(S, D^*)$, let $\sigma_{\mu \eta}: R \rightarrow R$ be the map defined  by
$$\sigma_{\mu \eta} (ds)=\mu_e(d) \eta(s)s$$
for $d \in D$ and $s =e \cdot s \in S$.  By extending $\sigma_{\mu \eta}$ linearly and utilizing the defining properties of $Z^1_{(\alpha, \xi)}(S, D^*)$, one can verify that $\sigma_{\mu \eta}$ is a ring homomorphism.  Then we have the following:

\begin{lemma}
There is a homomorphism $\hat{\Lambda}:  Z^1_{(\alpha, \xi)} (S, D^*) \rightarrow \mbox{Aut } R$.
\end{lemma}

\begin{proof}
For $( \mu, \eta) \in  Z^1_{(\alpha, \xi)} (S, D^*)$ and $ds \in R \cong D^{\alpha}_{\xi} S$, we easily see
\begin{align*}
\hat{\Lambda}( (\mu, \eta)) \circ \Lambda( ( \hat{\mu}, \hat{\eta}))(ds) &=
\hat{\Lambda}( (\mu, \eta))( \hat{\mu}(d) \hat{\eta}(s)(s)) \\ \\
&=\mu( \hat{\mu}(d) \hat{\eta}(s) ) \eta(s) s \\ \\
&=(\mu \circ \hat{\mu})(d)  \mu(\hat{\eta}(s) ) \eta(s) s =\hat{\Lambda}( (\mu, \eta) (\hat{\mu}, \hat{\eta}))(ds)
\end{align*}
\end{proof}

\begin{lemma}
There exists a monomorphism 
$$ \Lambda: H^1_{(\alpha, \xi)} (S, D^*) \rightarrow \mbox{Out } R$$
defined by
$$ \Lambda:( B^1( S, D^*)) ( \mu, \eta) \rightarrow (\mbox{Inn } R )\sigma_{\mu \eta}$$
for every $( \mu, \eta) \in Z^1_{(\alpha, \xi)} (S, D^*)$.
\end{lemma}

\begin{proof}
It is enough to show for  $(\mu, \eta) \in Z^1_{(\alpha, \xi)} (S, D^*)$ that $(\mu, \eta) \in B^1_{(\alpha, \xi)} (S, D^*)$ if and only if $\sigma_{\mu \eta} \in \mbox{Inn } R$.  First suppose $(\mu, \eta) \in B^1_{(\alpha, \xi)}(S, D^*) $.  Then there exists $\epsilon \in F^0 (S, D^*)$ such that 
\begin{align*}
& \mu_e (d)=\rho_{\epsilon(e)} (d) \\ \\
& \eta(s)= \epsilon(e) \alpha_s (\epsilon(f)^{-1})
\end{align*}
for $d \in D$ and $s=e \cdot s \cdot f \in S$.   Let $r=\sum_{e \in E} \epsilon(e)e \in R$.  Then $r^{-1}=\sum_{f \in E} \epsilon(f)^{-1}f \in R$ and
\begin{align*}
\sigma_{\mu \eta} (d s)=\mu_e(d) \eta(s) s& =\rho_{\epsilon(e)} (d)   \epsilon(e) \eta(s) \alpha_s (\epsilon(f)^{-1}) s \\ \\
\mu_e(d) \eta(s) s& =\rho_{\epsilon(e)} (d)   \epsilon(e) \eta(s) s \epsilon(f)^{-1}  \\ \\
&= \epsilon(e)  ds \epsilon_f^{-1} \\ \\
&=  \left( \sum_{e \in E} \epsilon(e)e \right) ds   \left( \sum_{f \in E} \epsilon^{-1} (f)f \right) \\ \\
&=\rho_{r} (ds)
\end{align*}
and it follows that $\sigma_{\mu \eta} \in \mbox{Inn } R$.

Now suppose $\sigma_{\mu \eta} \in \mbox{Inn } R$.  Then for all $d \in D^*$ and $s=e \cdot s \cdot f \in S$, there exists $r \in R$ such that
\begin{align*}
\sigma_{\mu \eta}(ds)= \rho_r (ds) & \Rightarrow \mu(d) \eta(s) s= r (ds) r^{-1}.
\end{align*}
Since $r \in R$, we have $r=\sum_{t \in S} d_t t$ and $r^{-1}= \sum_{t \in S} \hat{d}_t t$.  But
$$ \mu(d) \eta(s) s= r (ds) r^{-1}$$
implies that 
\begin{align*}
 \eta(s)s=r (s) r^{-1}&=\left(\sum_{t \in S} d_t t \right)s \left( \sum_{s \in S} \hat{d}_t t \right) \\ \\
 &=d_e  s \hat{d}_f  = d_e \alpha_s (\hat{d}_f) s.
\end{align*}
Since $\sigma_{\mu \eta} \in \mbox{Aut }R$ and $\sigma_{\mu \eta} (e)=\sigma_{\mu \eta}(e^2)$ implies
$$\eta(e)=\eta(e) \alpha_e( \eta(e))=\eta(e) \eta(e) \Rightarrow \eta(e)=1,$$ 
setting $s=e$ from above gives
$$1=d_e \hat{d}_e \Rightarrow \hat{d}_e=d_e^{-1}$$
for all $e \in E$.
Define $\epsilon \in F^0( S, D^*)$ by
$$ \epsilon(e)=d_e$$
for $e \in E$.  By the work above, 
\begin{equation}\label{inner and boundary1}
\eta(s)= \epsilon(e) \alpha_s (\epsilon(f)^{-1})
\end{equation}
 and
\begin{align*}
\mu(d) \eta(s) s= r (ds) r^{-1}\Rightarrow \mu(d) d_e \alpha_s (d_f^{-1})s= d_e d \alpha_s (d_f^{-1}) s
\end{align*}
from which we have
\begin{equation}\label{inner and boundary2}
\mu(d)=d_e d d_e^{-1}= \rho_{\epsilon(e)} (d).
\end{equation}
Equations \ref{inner and boundary1} and \ref{inner and boundary2} show that $( \mu, \eta) \in B^1 (S, D^*)$.
\end{proof}

For the idempotent set $E \subset S$, define $\mbox{Aut}_0 R$ by 
$$\mbox{Aut}_0 R=\{ \gamma_0 |~ \gamma_0(E)=E\}.$$
For any $\gamma \in \mbox{Aut}(R)$, it follows (by Lemma 2.4 of \cite{And-D'Amb}) there exists a normal automorphism $\gamma_0$ of $R$ with $(\mbox{Inn }R) \gamma=(\mbox{Inn }R)\gamma_0$.  If in addition, $\rho_r \in \mbox{Inn }R \cap \mbox{Aut}_0 R$, we must have $\rho_r (E)=E$.  If $R$ is basic, since $R e \cong R \rho_r (e)$ for all $e \in E$, we must have $\rho_r(e)=e$ for all $e \in E$.  It follows if $(\mbox{Inn }R) \gamma_0=(\mbox{Inn }R) \gamma'_0$ with $\gamma_0, ~\gamma'_0 \in \mbox{Aut}_0 R$ then $\gamma_0(e)=\gamma'_0(e)$ for all $e \in E$.  By Lemma \ref{goodmap}, we have
$$ \gamma_0(e)=\phi(e)=\gamma'_0(e)$$
for some $\phi \in \mbox{Aut }E$.  Since $S$ is square-free, $\phi$ is also an automorphism of $S$ and we have a well-defined map
$$\Phi: \mbox{Out } R \rightarrow \mbox{Aut }S$$
given by
$$ (\mbox{Inn }R) \gamma_0 \rightarrow \phi,$$
where $\gamma_0 \in \mbox{Aut}_0R$ such that $\gamma_0(ds)=\mu_e(d) \eta(s) \phi(s)$ for $d \in D$ and $s=e \cdot s \in S$.  A straightforward check shows that $\Phi$ is a group homomorphism.

\begin{lemma}
$\mbox{Im } \Lambda = \mbox{Ker } \Phi$
\end{lemma}

\begin{proof}
Let $( B^1( S, D^*)) ( \mu, \eta)  \in H^1_{(\alpha, \xi)} (S, D^*)$.  Then
$$ \Phi \Lambda (( B^1( S, D^*)) ( \mu, \eta))=\Phi( (\mbox{Inn}(R))\sigma_{\mu \eta})=1_S.$$
It follows then that $\mbox{Im } \Lambda \subseteq \mbox{Ker } \Phi$.  

Now let $\gamma \in\mbox{Aut}_0R$ with $(\mbox{Inn } R)\gamma \in \mbox{Ker } \Phi$.  Then by Lemma \ref{goodmap}, for each $s \in S$;
$$ \gamma(s)= \eta(s)s.$$
Because $\gamma$ is a homomorphism, $\gamma( s t)=\gamma(s) \gamma(t)$ and it follows that 
\begin{equation}\label{gammahomom}
\mu_e( \xi (s, t))\eta( s \cdot t)= \eta(s) \alpha_s (\eta(t)) \xi(s ,t).
\end{equation}
Also, $\gamma(s (dt))=\gamma(s) \gamma(dt)$ and from Equation \ref{gammahomom} we have
\begin{equation}
\mu_e \circ \alpha_s \circ \mu_f^{-1}=\rho_{\eta(s)} \circ \alpha_s.
\end{equation}
From the above equations we see $( \mu,  \eta)$ is a 1-cocycle and we have
$$ \mbox{(Inn $R$)} \gamma= \mbox{(Inn $R$)}\sigma_{\mu \eta}=\Lambda((B^1(S, D^*) (\mu, \eta)) \in \mbox{Im } \Lambda.$$
\end{proof}

We now focus on finding the image of $\Lambda$ in $\mbox{Aut }S$.   Recall that $\mbox{Aut }S$ acts on $\thin$.  For any element $[\alpha, \xi] \in \thin$, the stabilizer of the element is the set
$$\mbox{Stab}_{ [\alpha, \xi]}(\mbox{Aut }S)= \{ \phi \in \mbox{Aut}(S)| ~~[\alpha, \xi]= [\alpha^{\phi}, \xi^{\phi}]\}.$$
This means there exists an element $(\mu, \eta) \in F^0(S, \mbox{Aut}(D)) \ltimes F^1(S, D^*)$ such that

\begin{equation}
\mu_e \circ \alpha_s \circ \mu_f^{-1}= \rho_{\eta(s)} \circ \alpha_s^{\phi}
\end{equation}
and
\begin{equation}
\mu_e( \xi(s, t))= \eta(s) \alpha_s( \eta(t)) \xi^{\phi}(s,t) \eta( s \cdot t)^{-1}
\end{equation}
for all $(s, t) \in S^{<2>}$.

\begin{lemma}
$\mbox{Im }\Phi=\mbox{Stab}_{ [\alpha, \xi]}(\mbox{Aut }S)$.
\end{lemma}

\begin{proof}
First, suppose $\phi \in \mbox{Aut }S$ and $(\mu, \eta) \in F^0(S, \mbox{Aut}(D)) \ltimes F^1(S, D^*)$ with
\begin{equation}
\mu_e \circ \alpha_s \circ \mu_f^{-1}= \rho_{\eta(s)} \circ \alpha_s^{\phi}
\end{equation}
and
\begin{equation}
\mu_e( \xi(s, t))= \eta(s) \alpha_s( \eta(t)) \xi^{\phi}(s,t) \eta( s \cdot t)^{-1}
\end{equation}
for all $(s, t) \in S^{<2>}$.  Define a map $\sigma_{\mu \eta \phi}: R \rightarrow R$ by
$$\sigma_{\mu \eta \phi}(ds)=\mu_e(d) \eta(s) \phi(s)$$
for $d \in S$ and $s=e \cdot s \in S$.  Clearly, $\sigma_{\mu \eta \phi}$ preserves addition.  We must show that is also preserves multiplication.  Let $d_1, d_2 \in D$ and $s=e \cdot s \cdot f,~~t=f \cdot t  \in S$ with $s \cdot t \neq \theta$.  So
\begin{align*}
\sigma_{\mu \eta \phi} (d_1 s)\sigma_{\mu \eta \phi} ( d_2 t)&=
\mu_e(d_1) \eta(s) \phi(s) \mu_f(d_2) \eta(t) \phi(t)\\ \\
& =\mu_e(d_1) \eta(s) \alpha^{\phi}_s (\mu_f(d_2) ) \alpha^{\phi}_s(  \eta(t) ) \phi(s) \phi(t)\\ \\
&=\mu_e(d_1) \rho_{\eta(s)} ( \alpha^{\phi}_s (\mu_f(d_2) ) ) \eta(s) \alpha^{\phi}_s(  \eta(t) ) \xi^{\phi}(s,t)  \phi(s \cdot t)\\ \\
&=\mu_e(d_1) [\rho_{\eta(s)} ( \alpha^{\phi}_s (\mu_f(d_2) ) )] \eta(s) \alpha^{\phi}_s(  \eta(t) ) \xi^{\phi}(s,t)  \eta(s \cdot t)^{-1} \eta(s \cdot t) \phi(s \cdot t)\\ \\
&=\mu_e(d_1) [\mu_e(  \alpha_s (d_2))] \mu_e (\xi (s, t)) \eta(s \cdot t) \phi(s \cdot t) \\ \\
&=\sigma_{\mu \eta \phi} (d_1 \alpha_s (d_2) \xi( s, t) s \cdot t)=\sigma_{\mu \eta \phi} (d_1 s d_2 t)
\end{align*}
and it follows that $\sigma_{\mu \eta \phi} \in \mbox{Aut}_0R$.
\vspace{.2cm}
\\ Conversely, now suppose that $\gamma \in \mbox{Aut}_0 R$ with $\phi=\Phi (\mbox{Inn}(R)) \gamma$.  Then by Lemma \ref{goodmap}, $\gamma=\sigma_{\mu \eta \phi}$ for some $\mu \in F^0(S, \mbox{Aut}(D))$ and $\eta \in F^1 (S, D^*)$.  Since $\gamma=\sigma_{\mu \eta \phi}$ is a homomorphism, $\sigma_{\mu \eta \phi}(s t)=\sigma_{\mu \eta \phi}(s)\sigma_{\mu \eta \phi}(t)$ for all $s=e\cdot s , t=f \cdot t\in S$ and
\begin{align*}
\sigma_{\mu \eta \phi}(s t)=\sigma_{\mu \eta \phi}(s)\sigma_{\mu \eta \phi}(t) & \Rightarrow
\sigma_{\mu \eta \phi}(\xi(s, t) s \cdot t)=\eta(s) \phi(s) \eta(t) \phi(t) \\ \\
& \Rightarrow \mu_e( \xi(s, t)) \eta(s \cdot t) \phi(s \cdot t)=\eta(s) \alpha^{\phi}(\eta(t) ) \xi^{\phi} (s, t) \phi(s \cdot t)
\end{align*}
Equating coefficients of $\phi(s \cdot t)$, it follows
\begin{equation}\label{right above}
 \mu_e( \xi(s, t))=\eta(s) \alpha^{\phi}_s(\eta(t) ) \xi^{\phi} (s, t)  \eta(s \cdot t)^{-1}
\end{equation}
Furthermore,
\begin{align*}
& \sigma_{\mu \eta \phi}(s \mu_f^{-1}(d)t)=\sigma_{\mu \eta \phi}(s)\sigma_{\mu \eta \phi}(\mu_f^{-1}(d)t)  \Rightarrow
\sigma_{\mu \eta \phi}(\alpha_s (\mu_f^{-1}(d) \xi(s, t) s \cdot t)=\eta(s) \phi(s) d \eta(t) \phi(t) \\ \\
& \Rightarrow
\mu_e(\alpha_s (\mu_f^{-1}(d)))[ \mu_e( \xi(s, t)) \eta(s \cdot t) ]  \phi(s \cdot t)=\eta(s) \alpha^{\phi}_s( d) \alpha^{\phi}_s( \eta(t)) \xi^{\phi}(s, t) \phi(s \cdot t)
\end{align*}
Equating coefficients and using Equation \ref{right above} above, it follows
\begin{align*}
& \mu_e(\alpha_s (\mu_f^{-1}(d)) [\eta(s) \alpha^{\phi}_s(\eta(t) ) \xi^{\phi} (s, t)]=\eta(s) \alpha^{\phi}_s( d) \alpha^{\phi}_s( \eta(t)) \xi^{\phi}(s, t)  \\ \\
& \Rightarrow \mu_e(\alpha_s (\mu_f^{-1}(d))) =\eta(s) \alpha^{\phi}_s( d) \eta(s)^{-1} \\ \\
& \Rightarrow  \mu_e \circ \alpha_s \circ \mu_f^{-1}=\rho_{\eta(s)} \circ \alpha^{\phi}_s.
\end{align*}
So $\phi \in \mbox{Stab}_{[\alpha, \xi]} (\mbox{Aut}(S))$ and $\mbox{Im } \Phi=\mbox{Stab}_{[\alpha, \xi]} (\mbox{Aut}(S))$.
\end{proof}

Finally, we have the main theorem of this section.

\begin{theorem}\label{SES}
Let $R$ be a square-free ring with associated semigroup $S$, division ring $D$ and normal $(\xi, \alpha) \in Z^2(S, D*)$ such that $R \cong D^{\alpha}_{\xi} S$.  The following sequence is exact:

\begin{align*}
1 \longrightarrow H^1_{(\alpha, \xi)} (S, D^*) \overset{\Lambda}{\longrightarrow } \mbox{Out }R
\overset{\Phi}{\longrightarrow } \mbox{Stab}_{(\alpha, \xi)} (\mbox{Aut } S) \longrightarrow 1.
\end{align*}
If $[\alpha, \xi]=[1_D, 1]$, then $\mbox{Stab}_{(\alpha, \xi)}(\mbox{Aut }S)=\mbox{Aut }S$ and the following sequence is exact:
\begin{align*}
1 \longrightarrow H^1_{(\alpha, \xi)} (S, D^*) \overset{\Lambda}{\longrightarrow } \mbox{Out } R
\overset{\Phi}{\longrightarrow } \mbox{Aut } S \longrightarrow 1.
\end{align*}
Moreover, this last sequence splits, and so in this case $\mbox{Out } R$ is isomorphic to 
$$H^1_{(\alpha, \xi)}(S, D^*) \ltimes \mbox{Aut }S.$$
\end{theorem}

\begin{proof}
By Lemma \ref{rightiso}, we may assume $(\alpha, \xi)$ is normal.  The previous lemmas of this section prove the statement regarding the first short exact sequence, so it enough to show the sequence splits in the case that $[\alpha, \xi]=[1_D, 1]$.  We can assume that $(\alpha, \xi)=(1_D, 1)$.  In this case $s (dt)=ds \cdot t \in R$ for all $d \in D$ and $s, t \in S$.  Given $\phi \in \mbox{Aut }S$, define the map $\sigma_{\phi}$ by
$$ \sigma_{\phi} (ds)=d \phi(s)$$
for all $d \in D$ and $s\in S$ and extend linearly.  An easy check shows that $\sigma_{\phi} \in \mbox{Aut}_0 R$.  Now define a map $\Psi: \mbox{Aut }S \rightarrow \mbox{Out } R$ by
$$\Psi: \phi \rightarrow (\mbox{Inn } R) \sigma_{\phi}$$
for all $\phi \in \mbox{Aut }S$.  A straightforward check shows $\Psi$ to be a homomorphism of groups with $\Phi \Psi=1_{S} \in \mbox{Aut } S$.  Then the sequence is split and 
$$ \mbox{Out }R \cong H^1_{(\alpha, \xi)} (S, D^*) \rtimes \mbox{Aut }S.$$
\end{proof}

As in the algebra case, we have the following immediate corollaries to Theorem \ref{SES}.

\begin{corollary}
Let $R \cong D^{\alpha}_{\xi} S$ be a square-free ring.  If $H^1_{(\alpha, \xi)}(S, D^*)=1$, then
$$ \mbox{Out } R \cong \mbox{Stab}_{[\alpha, \xi]} (\mbox{Aut }S).$$
Moreover, if $[\alpha, \xi]=[1_D, 1]$, then
$$ \mbox{Out } R \cong\mbox{Aut }S.$$
\end{corollary}

\begin{corollary}
Let $R \cong D^{\alpha}_{\xi} S$ be a square-free ring.  Then every automorphism of $R$ is inner if and only if
$$ H^1_{(\alpha, \xi)}(S, D^*)= \mbox{Stab}_{[\alpha, \xi]} (\mbox{Aut }S)=1.$$
\end{corollary}

\begin{example}
Let D be a division ring that has an automorphism $\gamma$ that is not inner and let $S$ be a square-free semigroup with the following Hasse diagram

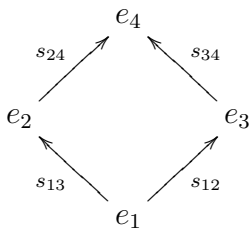
\begin{center} 
\vspace{.1cm}
 \xymatrix{		& \hspace{1cm} &		&	& e_{4}     \\
 \hspace{.2cm}  & \hspace{1cm}&		& e_{2}  \ar[ur]^{s_{24}}	&   		&  e_{3}   \ar[ul]_{s_{34}} \\ 
 			& \hspace{1cm}&   		&  		&  e_{1}  \ar[ur]_{s_{12}} \ar[ul]^{s_{13}}	}
\end{center}
and multiplication defined by $s \cdot t \neq \theta$ only if $s$ or $t$ in $E$.  Define $\alpha \in F^1(S, \mbox{Aut}(D))$ by 
$$ \alpha_s=\begin{cases} \gamma   & \quad \mbox{if} \quad s=s_{34}\vspace{.2cm}   \cr  1 & \quad \mbox{otherwise} 
\end{cases}$$
and take $\xi=1$.  For $s=s_{ij}$ denote the maps $\alpha_s$ and $\mu_s$ by $\alpha_{ij}$ and $\mu_{ij}$ respectively.

By Corollary 3.5 of \cite{D'Ambrosia}, $R =D^{\alpha}_{\xi} S$ is a square-free ring not isomorphic to $D \otimes_K A$, where $K$ is the center of $D$ and $A$ is a square-free $K$-algebra.

Both $Z^1_{(\alpha, \xi)}(S, D^*)$ and $B^1_{(\alpha, \xi)}(S, D^*)$ can be determined using the semigroup structure of $S$ and the defining properties of each group.  For $(\mu, \eta) \in Z^1_{(\alpha, \xi)}(S, D^*)$ we must have $\eta(e_i)=1$ for $i=1, 2, 3 ,4$ and no restrictions on $\eta(s)$ for $s \notin E$.  For $\mu$, the maps are determined by the idempotent $e$ for any $s=e \cdot s \cdot f \in S$.  But since $\alpha_s=1$ for all but one $s \in S$, the identity $\mu_j=\rho_{ \mu_i( \eta(s_{ij})}$ implies that $\mu$ is determined by $\mu_1$ and our previous choices for $\eta(s)$.  Then 
$$Z^1_{(\alpha, \xi)}(S, D^*) = \mbox{Aut }D \ltimes \Pi_{i=1}^4 D^*.$$

A similar procedure for $(\mu, \eta) \in B^1_{(\alpha, \xi)}(S, D^*)$ shows that $\eta$ is determined by $\epsilon: E \rightarrow D^*$ and this determines $\mu$ as well since $\mu_{ij}=\rho_{\epsilon(e_i)}$.  Therefore, 
$$B^1_{(\alpha, \xi)}(S, D^*) = \Pi_{i=1}^4 D^*.$$
and $$H^1_{(\alpha, \xi)}(S, D^*) = \mbox{Aut }D .$$
A tedious but straightforward check shows that $\mbox{Stab}_{[\alpha, \xi]}(\mbox{Aut }S)=\mbox{Aut }S \cong \mathbb{Z}_2$.  Hence, to $R$ we may associate the following short exact sequence
\begin{align*}
1 \longrightarrow \mbox{Aut }D \longrightarrow  \mbox{Out }R
\longrightarrow  \mathbb{Z}_2 \longrightarrow 1.
\end{align*}

\end{example}

\section*{Acknowledgement}

The author would like to thank Professor Frank Anderson for his continued guidance and inspiration.


\begin{thebibliography}{10}

\bibitem{Anderson}
{\sc F. W.~ Anderson}, Transitive square-free algebras, \textit{J. Algebra} \textbf{62} (1980), 61-85.
\\
\bibitem{And-D'Amb}
{\sc F.W.~Anderson and B.K.~D'Ambrosia}, Square-free algebras and their automorphism groups, \textit{Comm. Alg.} , \textbf{24} (1996) (10), 3163-3191.
\\
\bibitem{frank} 
{\sc F.W. Anderson and K.R. Fuller}, \textit{Rings and Categories of Modules} (2d ed.), Springer-Verlag, New York-Berlin-Heidelberg, 1992.
\\
\bibitem{Baclawski} 
{\sc K. Baclawski}, Automorphisms and derivations of incidence algebras, \emph{Proc. Amer. Math. Soc.},  \textbf{36} (1972), 351-356.
\\
\bibitem{Clark} 
{\sc W.E. Clark}, Twisted matrix units semigroup algebras, \textit{Duke Math J.}, Sept. (1967) 417-423.
\\
\bibitem{Coelho} 
{\sc S. P. Coelho}, The automorphism group of a structural matrix algebra, \emph{Linear Algebra Appl.},  \textbf{195} (1993), 35-58.
\\
\bibitem{D'Ambrosia}
{\sc B.K.~D'Ambrosia}, Square-free rings,  \textit{Comm. Alg.}, 27(5) (1999), 2045-2071.
\\
\bibitem{Dedecker}
{\sc P. Dedecker}, Three dimensinoal non-abelian cohomology for groups,  \textit{1969 Category Theory, Homology Theory and their Applications, II (Battelle Institute Conferene, Seattle, Wash., Vol. II)} 32-64 Springer, Berlin 1968.
\\
\bibitem{Mac Lane} 
{\sc S. Mac Lane}, \textit{Homology}, Springer-Verlag, New York-Berlin-Heidelberg, 1995.
\\
\bibitem{Scharlau} 
{\sc W. Scharlau}, Automorphisms and involutions of incidence algebras, in ``Representations of Algebras, Ottawa, 1974",  \textit{Lecture Notes in Math. No. 488}, Springer-Verlag, New York-Berlin-Heidelberg, 1975.
\\
\bibitem{sklar1} 
{\sc J. K. Sklar},  Binomial algebras,  \emph{Comm. Alg.} , \textbf{30}(4) (2002), 1961-1978.
\\
\bibitem{sklar2} 
{\sc J. K. Sklar},  Binomial Rings,  \emph{Comm. Alg.} , \textbf{32}(4) (2004), 1385-1399.
\\
\bibitem{Stanley} 
{\sc R.P. Stanley}, Structure of incidence algebras and their automorphism groups, \emph{Bull. Amer. Math. Soc.},  \textbf{76} (1970), 1236-1239.




\end{thebibliography}
\end{document}